\documentclass[11pt]{article}
\usepackage{amsmath,amssymb,amsthm}
\usepackage{graphicx,enumerate,a4wide}
\usepackage{pdfsync,epstopdf,comment,url}
\usepackage[x11names,usenames,dvipsnames,svgnames,table,rgb]{xcolor}
\numberwithin{equation}{section}
\newtheorem{theorem}{Theorem}[section]

\newtheorem{proposition}{Proposition}[section]
\newtheorem{lemma}{Lemma}[section]
\theoremstyle{definition}
\newtheorem{remark}{Remark}[section]
\newtheorem{definition}{Definition}[section]

\newtheorem{assumption}{Assumption}

\newcommand{\ie}{{i.e.}, }
\newcommand{\eg}{{e.g.}, }
\newcommand{\cf}{{cf.\ }}

\newcommand{\nn}{\mathbb{N}} 
\newcommand{\norm}[1]{\left\Vert {#1} \right\Vert} 
\newcommand{\erl}{\left(-\infty , +\infty\right]} 
\newcommand{\dom}[1]{\mathrm{dom}\,{#1}} 
\newcommand{\idom}[1]{\mathrm{int\,dom}\,{#1}} 
\newcommand{\cdom}[1]{\overline{\mathrm{dom}\,{#1}}} 
\newcommand{\gr}[1]{\mathrm{graph}\,{#1}} 
\newcommand{\crit}[1]{\mathrm{crit}\,{#1}}

\newcommand{\dist}{\mathrm{dist}} 
\newcommand{\act}[1]{\left\langle {#1} \right\rangle} 
\newcommand{\seq}[2]{\{{#1}_{{#2}}\}_{{#2} \in \mathbb{N}}}
\newcommand{\Seq}[2]{\{{#1}^{{#2}}\}_{{#2} \in \mathbb{N}}}

\newcommand{\limit}[2]{\lim_{{#1} \rightarrow {#2}}}
\newcommand{\argmin}{\mathrm{argmin}}

\newcommand{\sgn}{\mathrm{sgn}}

\newcommand{\barc}{\overline{C}}

\newcommand{\AAA}{\mathcal{A}}

\newcommand{\HHH}{\mathcal{H}}

\newcommand{\PPP}{\mathcal{P}}
\newcommand{\SSS}{\mathcal{S}}

\newcommand{\real}{\mathbb{R}} 
\newcommand{\rr}{\mathbb{R}} 

\title{First Order Methods beyond Convexity and Lipschitz Gradient Continuity with Applications to Quadratic Inverse Problems}

\author{J\'{e}r\^{o}me Bolte\footnote{TSE  (Universit\'{e} Toulouse I), Manufacture des Tabacs, 21 all\'{e}e de Brienne, 31015 Toulouse, France. E-mail: jerome.bolte@tse-fr.eu.} \and Shoham Sabach\footnote{Faculty of Industrial Engineering, The Technion, Haifa, 32000, Israel. E-mail: ssabach@ie.technion.ac.il.}\and Marc Teboulle\footnote{School of Mathematical Sciences, Tel-Aviv University, Ramat-Aviv 69978, Israel. E-mail: teboulle@post.tau.ac.il. This research was partially supported by the Israel Science Foundation, under ISF Grants  998-12 and 1844-16} \and Yakov Vaisbourd\footnote{School of Mathematical Sciences, Tel-Aviv University, Ramat-Aviv 69978, Israel. E-mail: yakov.vaisbourd@gmail.com. This research is supported by a postdoctoral fellowship under ISF grants 998-12 and 1844-16}}
\date{\today}

\begin{document}
\maketitle

	\begin{abstract}
		We focus on nonconvex and nonsmooth minimization problems with a composite objective, where the differentiable part of the objective is {\em freed} from the usual and restrictive global Lipschitz gradient continuity assumption. This longstanding smoothness restriction is pervasive in  first order methods (FOM), and was recently circumvent for  {\em convex composite} optimization by Bauschke, Bolte and Teboulle,  through a simple and elegant framework which captures, all at once, the geometry of the function and of the feasible set. Building on this work, we tackle genuine {\em nonconvex} problems. We first complement and extend their approach to derive a full extended descent lemma by introducing the notion of {\em  smooth adaptable functions}. We then consider a Bregman-based proximal gradient methods for the nonconvex composite model with smooth adaptable functions, which is proven to globally converge to a critical point under natural assumptions on the problem's data. To illustrate the power and potential of our general framework and results, we consider a broad class of {\em quadratic inverse} problems with sparsity constraints which arises in many fundamental applications, and we apply our approach to derive new globally convergent schemes for this class.
	\end{abstract}

	\noindent {\bfseries 2010 Mathematics Subject Classification:} Primary 90C25; Secondary 26B25, 49M27, 52A41, 65K05.

	\noindent {\bfseries Keywords:} Composite nonconvex nonsmooth minimization, proximal-gradient algorithms, descent lemma, Non Euclidean distances, Bregman distance, global convergence, Kudyka-{\L}osiajewicz proeprty, semi-algebraic functions, quadratic inverse problems, phase retrieval.

\section{Introduction} \label{Sec:Intro}
	The gradient method, forged by Cauchy about 170 years ago, is still at the heart of many fundamental schemes in modern computational optimization through many variants and relatives known as First Order Methods (FOM). These are currently the leading algorithms for solving large scale problems to medium accuracy. The focus and intensity of today research is mostly in the {\em convex} case: it goes from complexity to decomposition methods, from Lagrangian approaches to stochastic variants. For an appetizer, see for instance \cite{cw05, palo-elda-10, SNW11} and references therein, as well as the recent comprehensive text of Bertsekas \cite{bert15} with many relevant up-to-date and annotated sources and references.
\medskip

	A crucial and standard assumption common to almost all FOM is that the gradient of the smooth part in a given objective function has to be globally Lipschitz continuous on the entire underlying space. This is a very restrictive assumption which has only been circumvented by line-search approaches and/or quite complex inner loops which are unavoidably distorting the efficiency and the complexity of the initial method.
\medskip

	Recently, Bauschke, Bolte and Teboulle (BBT) \cite{BBT2016} solved this longstanding issue and dealt with non globally Lipschitz continuous gradient. Their  framework is simple and far reaching, it is based on adapting the geometry to the objective through the Bregman distance paradigm. It allows them to derive a new   Descent Lemma, whereby the usual upper quadratic approximation of a smooth function is replaced by a more general proximity measure which captures, all at once, the geometry of the function and the one of the feasible set. The corresponding FOM derived in \cite{BBT2016} come with guaranteed complexity estimates and pointwise global convergence results for {\em convex composite} minimization. Very recently, the publication \cite{quang} developed independently similar ideas for analyzing the convergence of a Bregman proximal gradient applied to the convex composite model in Banach spaces.
\medskip

	This work is a major departure from the current convex composite model and hence with far reaching consequences. Our main goal is to extend the BBT framework to the nonconvex setting. We complement and extend their approach to study {\em nonconvex composite} minimization problems. These consist in minimizing the sum of {\em two nonconvex} functions: an extended valued function $f$ and a continuously differentiable function $g$
	\begin{equation*}
		(\PPP) \qquad \inf \left\{ f\left(x\right) + g\left(x\right) : \; x \in \barc \right\},
	\end{equation*}
	where $\barc$ denotes the closure of $C$ which is a nonempty, convex and open set in $\real^{d}$ (see Section \ref{Sec:Prob} for a more precise statement).
\medskip

	Thus, here not only the functions $f$ and $g$ are not convex, but we also consider problems where $g$ {\em does not} admit a globally Lipschitz continuous gradient. The usual restrictive requirement of Lipschitz continuity of the gradient of $g$ in problem $(\PPP)$ is not needed, and replaced by a convexity condition which adapts to the geometry of the function $g$. To solve problem $(\PPP)$, we focus on a proximal-based gradient method which involves the non Euclidean distances of Bregman type, and which covers standard proximal-based gradient methods. We derive a class of proximal-based gradient algorithms which are proven to globally converge to a critical point of $(\PPP)$ when $C = \real^{d}$ and when the data is semi-algebraic. Note that one of the earliest work describing and analyzing the {\em classical} proximal gradient method in the nonconvex setting goes back to the work of Fukushima and Milne \cite{fuku81}. See also \cite{ABS, BST2014} for a full convergence analysis in the semi-algebraic setting, but which also imposed the usual restrictive global Lipschitz continuity of the gradient of $g$. These results are significantly improved in the present paper within the non Euclidean proximal framework.
\medskip

	To illustrate the power of our general framework and results, we consider a broad class of {\em quadratic inverse problems} with sparsity constraints which arise in many fundamental applications (see, for instance, \cite{BY12} and reference therein). We apply our approach to derive new and simple provably convergent schemes, which to the best of our knowledge are the first globally convergent algorithms for this important class of problems.
\medskip

   {\bf Outline of the paper.} The paper is organized as follows. We first complement and extend the BBT approach to derive a full Descent Lemma by introducing the notion of {\em smooth adaptable functions}, see Section \ref{Sec:Smo}. The following section presents the mathematical tools necessary to handle the nonconvex setting. It describes the problem and the corresponding proximal- based gradient method for the nonconvex composite model, freed from the usual global Lipschitz gradient continuity restriction. The  analysis is developed in Section \ref{Sec:Analysis} where the resulting scheme is proven to globally converges to a critical point under natural assumptions on the problem's data when $C = \real^{d}$. We demonstrate the potential of our framework by showing in Section \ref{Sec:Appl} how it can be successfully applied to a broad class of quadratic inverse problems with sparsity constraints, resulting in computationally simple and explicit iterative formulas. To make the paper self contained, we provide an appendix (see Appendix \ref{A:GlobConv}) which includes the relevant material and results for the convergence analysis of algorithm in the semi-algebraic setting.
\medskip

{\bf Notation.} We use standard notation and concepts which, unless otherwise specified can all be found in \cite{RW98}.

\section{Smooth Adaptable Functions and a Descent Lemma} \label{Sec:Smo}
	We begin by defining the notion of {\em smooth adaptable} functions. This notion is motivated by the recent work \cite{BBT2016} in which a {\em Lipschitz-like/Convexity condition} was introduced. This condition allows to lift the usual smoothness of the gradient of a given convex function, and derive  a {\em one sided} Descent Lemma, whereby the usual upper quadratic approximation of $C^{1}$ functions is given in terms of a more general proximity measure. Here, we extend and complement this notion by accommodating differentiable functions which are not necessarily convex, and we also derive a natural {\em two sided} Descent Lemma.

\subsection{Preliminaries on Proximity Measures}
   We first introduce our setting with some notations and definitions that will be used throughout the paper.
\medskip

	\begin{definition}(Kernel Generating Distance) \label{def:kgd}
    		Let $C$ be a nonempty, convex and open subset of $\real^{d}$. Associated with $C$, a function $h : \real^{d} \rightarrow \erl$ is called a kernel generating distance if it satisfies the following:
    		\begin{itemize}
    			\item[$\rm{(i)}$] $h$ is proper, lower semicontinuous and convex, with $\dom h \subset \barc$ and $\dom \partial h = C$.
     		\item[$\rm{(ii)}$] $h$ is $C^{1}$ on $\idom h \equiv C$.
 		\end{itemize}
		We denote the class of kernel generating distances by $\mathcal{G}(C)$.
 	\end{definition}
\medskip
	
	Given $h \in \mathcal{G}(C)$, define the proximity measure $D_{h} : \dom h \times \idom h \to \real_{+}$
    \begin{equation*}
   		D_{h}\left(x , y\right) := h\left(x\right) - \left[h\left(y\right) + \act{\nabla h\left(y\right) , x - y}\right].
   	\end{equation*}
	The proximity measure $D_{h}$ is the so-called {\em Bregman Distance} \cite{Bregman}. It measures the proximity between $x$ and $y$. Indeed, thanks to the gradient inequality, one has
	\begin{equation*}
		h \; \text{is convex if and only if}\;  D_{h}(x,y)\geq 0, \forall x \in \dom h,  y \in \idom h.
	\end{equation*}
 	If in addition $h$ is strictly convex, equality holds if and only if $x = y$. However, note that $D_{h}$ is not symmetric in general, unless $h$ is the energy kernel $h = \left(1/2\right)\norm{\cdot}^{2}$, which yields the classical squared Euclidean distance. For early foundation papers and key results on Bregman distances, associated proximal-based algorithms, as well as many examples of kernels $h$, generating Bregman distances, we refer the reader to \cite{CZ1992, tebo92, ChenTeb, Eck93, BB}.
\medskip

	Note that the structural form of $D_{h}$ is also useful when $h$ {\em is not} convex. It measures the discrepancy or error, between the value of $h$ at a given point $x \in \dom h$ from its linear approximation around $y \in \idom h$. In that case, obviously, the distance like property $D_{h}\left(x , y\right) \geq 0$ and equal to zero if and only if $x = y$, is no longer valid. However, $D_{h}$ still enjoys two simple, but remarkable properties, which follows from elementary algebra:
	\begin{itemize}
		\item {\bf The three point identity} \cite[Lemma 3.1]{ChenTeb} For any $ y , z \in \idom h$ and $x \in \dom h$,
			\begin{equation*}
				D_{h}\left(x , z\right) - D_{h}\left(x , y\right) - D_{h}\left(y , z\right) =
\langle \nabla h\left(y\right) - \nabla h\left(z\right), x-y \rangle.
			\end{equation*}
		\item {\bf Linear Additivity} For any $\alpha, \beta \in \real$, and any functions $h_{1}$ and $h_{2}$
		we have:
   			\begin{equation*}
   				D_{\alpha h_{1} + \beta h_{2}}\left(x , y\right) = \alpha D_{h_{1}}\left(x , y\right) + \beta D_{h_{2}}\left(x , y\right),
   			\end{equation*}
			for all couple $(x,y)\in \left(\dom h_1\cap \dom h_2\right )^2$ such that both $h_1$ and $h_2$ are differentiable at $y$.	
	\end{itemize}
	
\subsection{Smooth Adaptivity and an Extended Descent Lemma} \label{SSec:Smod}
	Throughout the paper we will work with a pair of functions $(g , h)$ satisfying:
	\begin{itemize}
		\item[$\rm{(i)}$] $h \in \mathcal{G}(C)$.
		\item[$\rm{(ii)}$] $g : \real^{d} \rightarrow \erl$ is a proper and lower semicontinuous function with $\dom h \subset \dom g$, which is continuously differentiable on $C = \idom h$
	\end{itemize}
	\begin{definition}[L-smooth adaptable] \label{def:smad}
		A pair $(g , h)$ is called $L$-smooth adaptable on $C$ if there exists $L > 0$ such that $Lh - g$ and $Lh + g$ are convex on $C$.
	\end{definition}
	Note that the above definition holds for any convex function $h$ which is $C^{1}$ on any open subset of $\dom h$. The additional properties required in the class $\mathcal{G}(C)$ are not necessary. Only the convexity of $Lh \pm g$ plays a central role. However, for the sake of consistency with the algorithmic development and results that follow, we will always use $h \in \mathcal{G}(C)$.
\medskip

	The above definition allows for immediately obtaining the promised {\em two-sided} Descent Lemma which naturally complements and extends the one derived in \cite[Lemma 1, p. 333]{BBT2016}.
	\begin{lemma}[Full Extended Descent Lemma] \label{L:NoLipsDescent}
		The pair of functions $(g , h)$ is $L$-smooth adaptable on $C$ if and only if:
		\begin{equation} \label{L:NoLipsDecent:1}
			\left| g\left(x\right) - g\left(y\right) - \act{\nabla g\left(y\right) , x - y} \right| \leq LD_{h}\left(x , y\right), \quad \forall \,\, x , y \in \idom h.
		\end{equation}
	\end{lemma}
	\begin{proof}
		By using Definition \ref{def:smad}, the pair $(g , h)$ is $L$-smooth adaptable on $C$ if and only if $Lh - g$ and $Lh +g$ are convex on $\idom h$. Therefore, thanks to the definition of the Bregman distance and its linear additivity property, this holds if and only if $LD_{h}\left(x , y\right) -D_{g}\left(x , y\right) = D_{Lh - g}\left(x , y\right) \geq 0$, and likewise for $D_{Lh +g}\left(x , y\right) \geq 0$, from which  the result  immediately follows.
	\end{proof}
	Clearly, in the setting of \cite{BBT2016}, that is, when the function $g$ is also assumed to be convex, the condition  $Lh + g$ is convex, trivially holds, and hence redundant. In this case, the Nolips Descent Lemma given in \cite[Lemma 1, p .333]{BBT2016} is recovered, \ie $D_{g}\left(x , y\right) \leq L D_{h}\left(x , y\right)$, though here $g$ needs not be convex.
\medskip

	Using the structural definition of $D_{g}$ (even though $D_{g}$ is not necessarily a proximity measure \`a la Bregman, since here $g$ is not convex) the full Descent Lemma reads compactly:
	\begin{equation*}
  		\left| D_{g}\left(x , y\right) \right| \leq LD_{h}\left(x , y\right), \quad  \forall \,\, x , y \in \idom h.
	\end{equation*}
	In the special case when the set $C$ is the whole space $\real^{d}$, and $h = \left(1/2\right)\norm{\cdot}^{2}$, the classical Descent Lemma for a function $g$ with an $L$-smooth gradient on $\real^{d}$, which provides lower and upper {\em quadratic} approximation for $g$, is recovered, \ie
	\begin{equation*}
		\left|D_{g}\left(x , y\right)\right| \equiv \left|g\left(x\right) - g\left(y\right) - \act{\nabla g\left(y\right) , x - y}\right| \leq \frac{L}{2}\norm{x - y}^{2}, \quad \forall \,\, x , y \in \real^{d}.
	\end{equation*}
	The new Descent Lemma allows for more general lower and upper approximations for $g$ by exploiting the geometry of the set $C$ through the use of the kernel function $h \in \mathcal{G}(C)$ and its associated Bregman proximity measure $D_{h}$.
\medskip

	\begin{remark}
  		The convexity requirement of $Lh + g$ (and therefore also the left-hand side inequality in \eqref{L:NoLipsDecent:1}) can be written with respect to a different parameter $\ell \leq L$. For simplicity, we have used here the same parameter $L$ which appears in the condition $Lh - g$.
	\end{remark}
	Similarly to \cite[Proposition 1, p. 334]{BBT2016}, it is easy to see that the convexity condition on $Lh - g$ (and likewise on $Lh +g$), admit various alternative reformulations which can facilitate its checking. In particular, when both $g$ and $h$ are  twice continuously differentiable on $C$, the usual Hessian test can be used, see the applications section for interesting cases.
\medskip

	For the purposes of this paper, and throughout the rest (unless otherwise specified) it will be enough to consider  only the condition $Lh -g$ convex, and its corresponding one-sided Descent Lemma. Thus, for convenience we adopt the following short hand terminology for such a pair $(g , h)$:

	\begin{center}
		\fcolorbox{black}{white}{\parbox{16cm}{\vspace{-0.1in}
		\center{\textbf{L-smad} holds on $C$ (SMooth ADaptable) if and only if}
			\begin{equation*}
				\exists \; L > 0\; \text{such that}\;  Lh - g \; \text{is convex on} \; \idom h \; \Longleftrightarrow \; D_{g}\left(x , y\right) \leq LD_{h}\left(x , y\right), \; \forall \; x , y \in \idom h. \vspace{-0.1in}
			\end{equation*}}}
	\end{center}	
	We end this section by observing that the \textbf{L-smad} property is {\em invariant} when $h$ is assumed $\sigma$-strongly convex, where $\sigma > 0$ stands for the strong convexity modulus, namely $h -\left(\sigma/2\right)\norm{x}^{2}$ is convex. Indeed,  since no convexity is needed/assumed for $g$, we obviously have with $\omega\left(\cdot\right) := \left(1/2\right)\norm{\cdot}^{2}$:
		\begin{equation*}
			Lh - g = L\left(h - \omega\right) - \left(g - L\omega\right) := L {\bar h} - {\bar g},
		\end{equation*}
		namely, \textbf{L-smad} holds on $C$ for the new pair $\left( {\bar g} , {\bar h}\right)$.

\section{The Problem and a Bregman Proximal Gradient Algorithm} \label{Sec:Prob}
	Our aim is to solve the nonconvex problem $(\PPP)$ with $C := \idom{h}$, that is,
	\begin{equation*}
		\inf \left\{ \Psi\left(x\right) \equiv f\left(x\right) + g\left(x\right) : \; x \in\barc \right\},
	\end{equation*}
	where the following assumptions on the problem's data are made throughout the paper.
	\begin{assumption} \label{A:AssumptionA}
		\begin{itemize}
    			\item[$\rm{(i)}$] $h \in \mathcal{G}(C)$ with $\barc = \cdom h$.
     		\item[$\rm{(ii)}$] $f : \real^{d} \rightarrow \erl$ is a proper and lower semicontinuous function with $\dom f \cap C \neq \emptyset$.
    			\item[$\rm{(iii)}$] $g : \real^{d} \rightarrow \erl$ is a proper and lower semicontinuous function with $\dom{h} \subset \dom{g}$, which is $C^1$ on $C$.
        		\item[$\rm{(iv)}$] $v(\PPP) := \inf \left\{ \Psi\left(x\right) : \; x \in \barc \right\} > -\infty$.
    		\end{itemize}	
	\end{assumption}

\subsection{A Bregman Proximal Gradient Algorithm} \label{sec:pgno}
	To present the algorithm for tackling problem $(\PPP)$, we define for all $x \in \idom h$ and any step-size $\lambda > 0$, the {\em Bregman proximal gradient} mapping
	\begin{align}
		T_{\lambda}\left(x\right) & := \argmin \left\{ f\left(u\right) + \act{\nabla g\left(x\right) , u - x} + \frac{1}{\lambda} D_{h}\left(u , x\right) : \, u \in \barc \right\} \nonumber \\
		& = \argmin \left\{ f\left(u\right) + \act{\nabla g\left(x\right) , u - x} + \frac{1}{\lambda} D_{h}\left(u , x\right) : \, u \in \real^{d} \right\}, \label{D:OperT}		 	
	\end{align}
	where the second inequality follows from the fact that $\dom h \subset \barc$. Note that, since $f$ is {\em nonconvex}, the mapping $T_{\lambda}$ is {\em not} in general single-valued. This map emerges from the usual approach which consists of linearizing the differential part $g$ around $x$, and regularize it with a proximal distance from that point. Clearly, with  $h \equiv \left(1/2\right)\norm{\cdot}^{2}$, the above boils down to the classical set-valued proximal gradient mapping.
\medskip

	We now briefly discuss the properties of the Bregman proximal mapping $T_{\lambda}$  in the nonconvex setting. For that purpose, we make the following two additional assumptions.
	\begin{assumption} \label{A:AssumptionB}
		The function $h + \lambda f$ is supercoercive for all $\lambda > 0$, that is,
		\begin{equation*}
			\lim_{\norm{u} \rightarrow \infty} \frac{h\left(u\right) + \lambda f\left(u\right)}{\norm{u}} = \infty.
		\end{equation*}		
	\end{assumption}
	\begin{assumption} \label{A:AssumptionC}
		For all $x \in C$ we have
		\begin{equation*}
			T_\lambda (x)\subset C, \: \forall x \in C.		
		\end{equation*}		
	\end{assumption}
	Assumption \ref{A:AssumptionB} is a quite standard coercivity condition, it is for instance automatically satisfied when $\barc$ is compact. On the other hand Assumption \ref{A:AssumptionC} can be shown to hold under a classical qualification condition which ensures the use of a partial sum rule \cite[Corollary 10.9, p. 430]{RW98} and which is formulated through the use of horizon subgradients
	\begin{equation}
		\partial^{\infty} f\left(x\right) \cap \left(-\partial^{\infty} h\left(x\right)\right) = \left\{ 0 \right\}, \quad \forall \,\, x \in \real^{d}.
	\end{equation}
	It also holds automatically when $f$ is convex or when $C = \real^{d}$. Another approach to warrant Assumption \ref{A:AssumptionC} is to consider extending the notion of prox-boundedness as defined in \cite[Chapter 1, p. 19-20]{RW98} to Bregman proximal-based map and their envelopes. However, to keep our presentation simple and transparent, these technical issues will not be pursued here, since they are irrelevant in the context of this paper.
\medskip

	We have the following basic result.
	\begin{lemma}[Well-Posedness of $T_{\lambda}$] \label{P:WellProximal}
		Suppose that Assumptions \ref{A:AssumptionA}, \ref{A:AssumptionB} and \ref{A:AssumptionC} hold, and let $x \in \idom h$. Then, the set $T_{\lambda}\left(x\right)$ is a nonempty and compact subset of $C$.
	\end{lemma}
	\begin{proof}
		Fix any $x \in \idom h$, and $\lambda >0$. For any $u \in \real^{d}$, we define the function
		 \begin{equation} \label{D:Psih}
		 	\Psi_{h}\left(u\right) := \lambda f\left(u\right) + \lambda \act{\nabla g\left(x\right) , u - x} + D_{h}\left(u , x\right),
		 \end{equation}
		 so that $T_\lambda\left(x\right) = \argmin_{u \in \real^{d}} \Psi_{h}\left(u\right)$. Using the definition of the Bregman distance, note that $\Psi_h$ can be rewritten as:
\begin{equation*}
		 	\Psi_{h}\left(u\right) = \lambda f(u) + h(u) + \act{\gamma , u} + \rho,
		\end{equation*}
		where $\gamma := \lambda \nabla g\left(x\right) - \nabla h\left(x\right) \in \real^{d}$ and $\rho := - h\left(x\right) - \act{\gamma , x} \in \real$ are constant quantities. We now show that $\Psi_{h}$ is level bounded on $\real^{d}$, \ie $\lim_{\norm{u} \rightarrow \infty} \Psi_{h}\left(u\right) = \infty$. Using the Cauchy-Schwarz inequality in the above definition of $\Psi_{h}$ we obtain
		 \begin{equation*}
		 	\Psi_{h}\left(u\right) \geq \lambda f\left(u\right) + h\left(u\right) - \norm{\gamma}\cdot\norm{u} - \left|\rho\right| = \norm{u}\left(\frac{\lambda f\left(u\right) + h\left(u\right)}{\norm{u}} - \norm{\gamma} - \frac{\left|\rho\right|}{\norm{u}}\right).
		 \end{equation*}
		Passing to the limit as $\norm{u} \rightarrow \infty$,  the supercoercivity of $h + \lambda f$ gives  $\lim_{\norm{u} \rightarrow \infty} \Psi_{h}\left(u\right) = \infty$. Therefore, since $\Psi_{h}$ is also proper and lower semicontinuous, invoking the modern form of Weierstrass' theorem (see, \eg \cite[Theorem 1.9, p. 11]{RW98}), it follows that the value $\inf_{\real^{d}} \Psi_{h}$ is finite, and the set $\argmin_{u \in \real^{d}} \Psi_{h}\left(u\right) \equiv T_{\lambda}\left(x\right)$ is nonempty and compact.
	\end{proof}	
	\begin{remark}[The case when $f$ is convex] \label{r:convexf}
		First note that in this case problem ($\PPP$) remains a {\em nonconvex} composite minimization which consists in minimizing the sum of a nonsmooth and convex function $f$ with a nonconvex and continuously differentiable function $g$. However, in this case, under our Assumption \ref{A:AssumptionA}, the function $\Psi_{h}$ which is proper and lower semicontinuous is now convex. It can then be shown (see \cite[Lemma 2, p. 336]{BBT2016}), that in this case the mapping $T_{\lambda}$ is also single-valued from $\idom h$ to $\idom h$.
 	\end{remark}
	We are now ready to describe our algorithm for solving the nonconvex composite problem $(\PPP)$.
\medskip

    \fcolorbox{black}{Ivory2}{\parbox{15cm}{{\bf Bregman Proximal Gradient - BPG} \\
    		{\bf Input.} A  function $h \in \mathcal{G}(C)$ with $C = \idom h$ such that \textbf{L-smad} holds on $C$. \\
    		{\bf Initialization.} $x^{0} \in \idom h$ and let $\lambda > 0$. \\
    		{\bf General Step.} For $k = 1 , 2 , \ldots$, compute
        		\begin{equation} \label{NoLipsPG}
            		x^{k}\in \argmin \left\{  f\left(x\right) + \act{x - x^{k - 1} , \nabla g\left(x^{k - 1}\right)} + \frac{1}{\lambda}D_{h}\left(x , x^{k - 1}\right) : \; x \in \barc \right\}.
	        	\end{equation}}}
\bigskip

	Under our standing Assumptions \ref{A:AssumptionA}, \ref{A:AssumptionB} and \ref{A:AssumptionC} the algorithm is well-defined by Lemma \ref{P:WellProximal}. In the next section we analyze its properties, and establish a global convergence result to a critical point of $\Psi$.

\section{Convergence Analysis of BPG} \label{Sec:Analysis}
	Throughout this section, we take the following as our blanket assumption
	\begin{itemize}
		\item[$\rm{(i)}$] \textbf{L-smad} holds on $C$.
		\item[$\rm{(ii)}$] Assumptions \ref{A:AssumptionA}, \ref{A:AssumptionB} and \ref{A:AssumptionC} hold.
	\end{itemize}

\subsection{Properties of the Algorithm}	
	Using the Descent Lemma (see Lemma \ref{L:NoLipsDescent}), we easily obtain the following key estimation for the composite objective function $\Psi$, which will play an essential role to derive our main convergence results.
	\begin{lemma}[Sufficient Decrease Property] \label{L:DescNcvx}
		For any $x \in \idom h$, and any $x^{+} \in \idom h$ defined by
        \begin{equation} \label{L:DescNcvx:1}
            x^{+} \in \argmin \left\{ f\left(u\right) + \act{u - x , \nabla g\left(x\right)} + \frac{1}{\lambda}D_{h}\left(u , x\right) : \; u \in \barc \right\}, \quad (\lambda > 0),
        \end{equation}
        we have
        \begin{equation} \label{L:DescNcvx:2}
            \lambda\Psi\left(x^{+}\right) \leq \lambda\Psi\left(x\right) - \left(1 - \lambda L\right)D_{h}\left(x^{+}  , x\right).
        \end{equation}
    \end{lemma}
    In particular, with $0 < \lambda L < 1$, the sufficient decrease in the composite objective function value $\Psi$ is ensured.
    \begin{proof}
    		By the definition of global optimality for \eqref{L:DescNcvx:1}, we obtain by taking $u = x \in \idom h$, that
        \begin{equation}  \label{L:DescNcvx:3}
       		f\left(x^{+}\right) + \act{x^{+} - x , \nabla g\left(x\right)} + \frac{1}{\lambda} D_{h}\left(x^{+} , x\right) \leq f \left(x\right).
        \end{equation}
        Invoking the Descent Lemma (see Lemma \ref{L:NoLipsDescent}) for $g$ we then get with the above,
        \begin{align*}
            g\left(x^{+}\right) + f\left(x^{+}\right) & \leq g\left(x\right) + \act{x^{+} - x , \nabla g\left(x\right)} + LD_{h}\left(x^{+} , x\right) + f\left(x^{+}\right) \\
            &\leq g\left(x\right) + LD_{h}\left(x^{+} , x\right) + f \left(x\right) - \frac{1}{\lambda}D_{h}\left(x^{+} , x\right)  \\
            & = g\left(x\right) + f\left(x\right) - \left(\frac{1}{\lambda} - L\right)D_{h}\left(x^{+} , x\right),
        \end{align*}
        and hence with $\Psi = f + g$, the desired inequality  for $\Psi$ is proved, and the last statement immediately follows with $0< \lambda L < 1$.
    \end{proof}
	\begin{remark}
		\begin{itemize}
			\item[$\rm{(i)}$] When $f$ is assumed convex, using the global optimality condition for the corresponding convex problem, which defines $x^{+}$ through \eqref{L:DescNcvx:1}, followed by using the three points identity property of a Bregman distance, and the \textbf{L-smad} property, one easily see that for any $x \in \idom h$, the inequality \eqref{L:DescNcvx:2} can be strengthened and reads
       			\begin{equation*}
            			\lambda\left(\Psi\left(x^{+}\right) - \Psi\left(u\right)\right) \leq D_{h}\left(u , x\right) - D_{h}\left(u , x^{+}\right) - \left(1 - \lambda L\right)D_{h}\left(x^{+} , x\right) - \lambda D_{g}(u,x), \quad \forall \,\, u \in \dom h.
       			\end{equation*}
 				When $g$ is also convex, then the last term $-\lambda D_{g}\left(u , x\right) \leq 0$, and the above result recovers the key estimation result proven in \cite[Lemma 5, p. 340]{BBT2016}.
			\item[$\rm{(ii)}$] When $g$ is {\em nonconvex} and the {\em Full Descent Lemma} (see Lemma \ref{L:NoLipsDescent}) holds, \ie there exists $L > 0$ with $Lh + g$ convex, then we have $- D_{g}\left(u , x\right) \leq L D_{h}\left(u , x\right)$ and hence the above inequality reduces to
	    			\begin{equation*}
            			\lambda\left(\Psi\left(x^{+}\right) - \Psi\left(u\right)\right) \leq \left(1 + \lambda L\right)D_{h}\left(u , x\right) - \left(1 - \lambda L\right)D_{h}\left(x^{+} , x\right) - D_{h}\left(u , x^{+}\right), \quad \forall \,\, u \in \dom h.
	        	\end{equation*}
		\end{itemize}    		
	\end{remark}
    The result given in Lemma \ref{L:DescNcvx} is valid for any $\lambda > 0$, and as seen, imposing the condition $0 < \lambda L < 1$ ensures a sufficient decrease in the composite objective function value $\Psi$. This fact yields the following result.
	\begin{proposition} \label{P:SuffDesc}
   		Let $\Seq{x}{k}$ be a sequence generated by BPG with $0 < \lambda L < 1$. Then the following assertions hold:
        \begin{itemize}
            \item[$\rm{(i)}$] The sequence $\left\{ \Psi\left(x^{k}\right) \right\}_{k \in \nn}$ is
            		nonincreasing.
            \item[$\rm{(ii)}$] $\sum_{k = 1}^{\infty} D_{h}\left(x^{k} , x^{k - 1}\right) < \infty$, and
            		hence the sequence $\left\{ D_{h}\left(x^{k} , x^{k - 1}\right)  \right\}_{k \in \nn}$ converges to zero.
           	\item[$\rm{(iii)}$] $\min_{1 \leq k \leq n} D_{h}\left(x^{k} , x^{k - 1}\right) \leq \frac{\lambda}{n}\left( \frac{\Psi\left(x^{0}\right) - \Psi_{\ast}}{1 - \lambda L}\right),$ where $\Psi_{\ast} = v(\PPP) >-\infty$ (by Assumption \ref{A:AssumptionA}(iv)).
		\end{itemize}    		
    \end{proposition}
    \begin{proof}
        \begin{itemize}
            \item[$\rm{(i)}$] Fix $k \geq 1$. Under our assumptions and using the iterative step \eqref{NoLipsPG}, we can apply Lemma \ref{L:DescNcvx} with $x = x^{k - 1}$ and $x^{+} = x^{k}$, to obtain
           		\begin{equation} \label{P:SuffDesc:1}
                		\lambda\left(\Psi\left(x^{k}\right) - \Psi\left(x^{k - 1}\right)\right) \leq -\left(1 - \lambda L\right)D_{h}\left(x^{k} , x^{k - 1}\right).
              	\end{equation}
               	Thus, with $0 < \lambda L <1$, we immediately obtain that the sequence $\left\{ \Psi\left(x^{k}\right) \right\}_{k \in \nn}$ is nonincreasing.
         	\item[$\rm{(ii)}$] Let $n$ be a positive integer. Summing the above inequality from $k = 1$ to $n$ we get
                \begin{equation} \label{P:SuffDesc:2}
                   	\sum_{k = 1}^{n} D_{h}\left(x^{k} , x^{k - 1}\right) \leq \frac {\lambda\left(\Psi\left(x^{0}\right) - \Psi\left(x^{n}\right)\right)}{1 - \lambda L} \leq \frac {\lambda\left(\Psi\left(x^{0}\right) - \Psi_{\ast}\right)}{1 - \lambda L},
                \end{equation}
	 			where $\Psi_{\ast} = v(\PPP) > -\infty$. Taking the limit as $n \rightarrow \infty$, we obtain the first desired assertion (ii), from which  we immediately deduce that $\left\{ D_{h}\left(x^{k} , x^{k - 1}\right)  \right\}_{k \in \nn}$ converges to zero.
        	\item[$\rm{(iii)}$] From \eqref{P:SuffDesc:2} we also obtain,
                \begin{equation*}
                   	n\min_{1\leq k \leq n} D_{h}\left(x^{k} , x^{k - 1}\right) \leq \sum_{k = 1}^{n} D_{h}\left(x^{k} , x^{k - 1}\right) \leq \frac{\lambda\left(\Psi\left(x^{0}\right) - \Psi_{\ast}\right)}{1 - \lambda L},
                \end{equation*}
        		which after division by $n$ yields the desired result.
        \end{itemize}
        \vspace{-0.2in}
    \end{proof}	
	Clearly, with $h$ being the energy kernel on $C \equiv \real^{d}$, the above proposition yields all the classical results for the nonconvex proximal gradient method (\eg set $\lambda L = 1/2$). In particular, in this case one obtains that the corresponding classical gradient mapping , defined by $G\left(x^{k - 1}\right) := \norm{x^{k - 1} - T_{\lambda}\left(x^{k - 1}\right)}$ (see \cite[Theorem 2.3, p. 61]{BT09}) converges to zero at a rate of $O(1/\sqrt{n})$.
\medskip

	Finally, note that thanks to the invariance of \textbf{$L$-smad} (cf. Section \ref{SSec:Smod}), we can assume that $h$ is $\sigma$-strongly convex on $C$. In that case, for any nonempty set $S \subset \real^{d}$, let $\dist(y, S) := \inf\left\{ \norm{u - y} : \; u \in S \right\}$. Since $x^{k} \in T_{\lambda}\left(x^{k - 1}\right)$, using the $\sigma$-strong convexity of $h$ combined  with Proposition \ref{P:SuffDesc}(iii) we immediately obtain the following rate of convergence result for two successive iterates:
  \begin{equation*}
      \min_{1 \leq k \leq n} \dist^{2}\left(x^{k - 1} , T_{\lambda}\left(x^{k - 1}\right)\right) \leq \min_{1 \leq k \leq n} \norm{x^{k} - x^{k - 1}}^{2} \leq \frac{1}{n}\cdot\frac{\lambda\left(\Psi\left(x^{0}\right) - \Psi_{\ast}\right)}{\sigma\left(1 - \lambda L\right)}.
   	\end{equation*}

\subsection{ Global Convergence of BPG}\label{ss:global}
	Throughout this section we consider problem $(\PPP)$ defined on $C \equiv \real^{d}$, namely
	\begin{equation*}
		(\PPP) \qquad v(\PPP) = \inf \left\{ f\left(x\right) + g\left(x\right) : \; x \in \real^{d} \right\}.
	\end{equation*}
	Here, throughout this subsection, we additionally assume that
	\begin{assumption} \label{A:AssumptionD}
		\begin{itemize}
			\item[$\rm{(i)}$] $\dom h = \real^{d}$ and $h$ is $\sigma$-strongly convex on $\real^{d}$.
			\item[$\rm{(ii)}$] $\nabla h$ and $\nabla g$ are Lipschitz continuous on any bounded subset of $\real^{d}$.
		\end{itemize}
	\end{assumption}
	Noting that Assumption \ref{A:AssumptionC} is automatically fulfilled since $C = \real^{d}$. As usual we use the concept of limiting subdifferential for $f$ so that thanks to Fermat's rule \cite[Theorem 10.1, p. 422]{RW98}, the set of critical points of $\Psi$ is given by:
	\begin{equation*}
		\crit \Psi = \left\{ x \in \real^{d} : \;  0 \in \partial \Psi\left(x\right) \equiv \partial f\left(x\right) + \nabla g\left(x\right) \right\}.
	\end{equation*}
	To prove the global convergence of the sequence $\Seq{x}{k}$ generated by BPG to a critical point of $\Psi$, we apply the methodology developed in \cite{BST2014}. For the reader's convenience, and to make this paper self-contained,  the main tools and relevant proofs are given in Appendix \ref{A:GlobConv}. First, we describe three key ingredients of the methodology \cite{BST2014}.
	\begin{definition}[Gradient-like descent sequence] \label{def-descent}
		Let $F: \rr^d \to (-\infty, \infty]$ be a proper and lower semicontinuous function. A sequence $\Seq{x}{k}$ is called \textit{a gradient-like descent sequence} for $F$ if the following three conditions hold:
		\begin{itemize}
        		\item[$\rm{(C1)}$] \textit{Sufficient decrease property.} There exists a positive scalar $\rho_{1}$ such that
            		\begin{equation*}
                		\rho_{1}\norm{x^{k + 1} - x^{k}}^{2} \leq F(x^{k}) - F(x^{k + 1}), \quad \forall \,\, k \in \nn.
	            \end{equation*}
    			\item[$\rm{(C2)}$] \textit{A subgradient lower bound for the iterates gap.} There exists a positive scalar $\rho_{2}$ such that         	
		   		\begin{equation*}
	    				\norm{w^{k + 1}} \leq \rho_{2}\norm{x^{k + 1} - x^{k}} \text{ for some } w^{k + 1} \text{ in }\partial F(x^{k + 1}), \quad \forall \,\, k \in \nn.
	            \end{equation*}
			\item[$\rm{(C3)}$] Let $\overline{x}$ be a limit point of a subsequence $\left\{ x^{k} \right\}_{k \in {\cal K}}$, then $\limsup_{k \in {\cal K} \subset \nn} F\left(x^{k}\right) \leq F\left(\overline{x}\right)$.
        \end{itemize}	
	\end{definition}
	The two conditions (C1) and (C2) are typical for any descent type algorithm (see, \eg \cite{AB2009}). They are also the main tools to prove subsequential convergence, as recorded below. The condition (C3) is a minimal and very weak continuity requirement, which, in particular, holds when $F$ is continuous. We denote by $\omega\left(x^{0}\right)$ the set of all limit points of $\Seq{x}{k}$.
	\begin{lemma}[Abstract subsequence convergence] \label{L:SubConv}
		Let $\Seq{x}{k}$ be a bounded gradient-like descent sequence for $F$. Then, $\omega\left(x^{0}\right)$ is a nonempty and compact subset of $\crit F$, and we have
		\begin{equation} \label{L:SubConv:1}
			\limit{k}{\infty} \dist\left(x^{k} , \omega\left(x^{0}\right)\right) = 0.
		\end{equation}
		In addition, the  function $F$ is finite and constant on $\omega\left(x^{0}\right)$.
	\end{lemma}
	\begin{proof}
		See Appendix \ref{A:GlobConv}.
	\end{proof}
	\begin{remark}[Boundedness of Sequences]
		Observe that the mere coercivity of $F$ ensures that any gradient-like sequence is bounded.
	\end{remark}
	The above, together with the so-called nonsmooth Kurdyka-{\L}ojasiewicz (KL) property (see \cite{BDL2006} and Appendix \ref{A:GlobConv} for details) allows us to establish our main convergence results of BPG.
	\begin{theorem}[Convergence of BPG] \label{T:GlobConv}
		Let $\Seq{x}{k}$ be a sequence generated by BPG which is assumed to be bounded and let $0 < \lambda L < 1$. The following assertions hold.
		\begin{itemize}
			\item[$\rm{(i)}$] \textbf{Subsequential convergence.} Any limit point of the sequence $\Seq{x}{k}$ is a critical point of $\Psi$.
			\item[$\rm{(ii)}$] \textbf{Global convergence.} Suppose that $\Psi$ satisfies the KL property on $\dom{\Psi}$. Then, the sequence $\Seq{x}{k}$ has finite length and converges to a critical point $x^{\ast}$ of $\Psi$.
		\end{itemize}
	\end{theorem}
	\begin{proof}
		(i) We first prove that the sequence $\Seq{x}{k}$ is a gradient-like descent sequence; namely it satisfies the three conditions (C1), (C2) and (C3) as described in  Definition \ref{def-descent}, and the result will then be established by invoking Lemma \ref{L:SubConv}.
\medskip

		From Proposition \ref{P:SuffDesc} (see \eqref{P:SuffDesc:1}) we have that
		\begin{equation} \label{T:GlobConv:1}
       		\Psi(x^{k})  - \Psi(x^{k + 1}) \geq \left(\frac{1}{\lambda} - L\right)D_{h}\left(x^{k + 1} , x^{k}\right) \geq \left(\frac{1}{\lambda} - L\right)\frac{\sigma}{2}\norm{x^{k + 1} - x^{k}}^{2},
        \end{equation}
		where the last inequality follows from the $\sigma$-strong convexity of $h$ (see Assumption \ref{A:AssumptionD}(i)). This proves that condition (C1) holds true.
\medskip

		Writing the optimality condition of the optimization problem which defines $x^{k + 1}$ (see \eqref{NoLipsPG}) yields that
		\begin{equation*}
			0 \in \partial f(x^{k + 1}) + \nabla g(x^{k}) + \frac{1}{\lambda}\left(\nabla h(x^{k + 1}) - \nabla h(x^{k})\right).
		\end{equation*}
		Therefore, by defining
		\begin{equation*}
			w^{k + 1} \equiv \nabla g(x^{k + 1}) - \nabla g(x^{k}) + \frac{1}{\lambda}\left(\nabla h(x^{k}) - \nabla h(x^{k + 1})\right),
		\end{equation*}
		we obviously obtain that $w^{k + 1} \in \partial \Psi\left(x^{k + 1}\right)$. Since $\Seq{x}{k}$ is a bounded sequence and both $\nabla h$ and $\nabla g$ are Lipschitz continuous on bounded subsets of $\rr^d$ (see Assumption \ref{A:AssumptionD}(ii)), there exists $M > 0$ such that
		\begin{equation*}
			\norm{w^{k + 1}} \leq \norm{\nabla g(x^{k + 1}) - \nabla g(x^{k})} + \frac{1}{\lambda}\norm{\nabla h(x^{k}) - \nabla h(x^{k + 1})} \leq M\left(1 + \frac{1}{\lambda}\right)\norm{x^{k + 1} - x^{k}}.
		\end{equation*}
		This proves that condition (C2) also holds true.
\medskip

		Consider a subsequence $\left\{ x^{n_{k}} \right\}_{n \in \nn}$ which converges to some point  $x^{\ast}$ (there exists such a subsequence since the sequence $\Seq{x}{k}$ is assumed to be bounded). Using \eqref{T:GlobConv:1} and Proposition \ref{P:SuffDesc}(iii) we obtain that $\lim_{k \rightarrow \infty}  \norm{x^{k} - x^{k - 1}} = 0$. Therefore, the sequence $\left\{ x^{n_k - 1} \right\}_{k \in \nn}$ also converges to $x^{\ast}$. In addition, since $h$ is continuously differentiable on $\real^{d}$ we have that $\lim_{k \rightarrow \infty}  D_{h}\left(x^{\ast} , x^{k - 1}\right) = 0$.  Now, from \eqref{NoLipsPG}, it follows that
		\begin{align*}
       		f(x^{k}) + \act{x^{k} - x^{k - 1} , \nabla g(x^{k - 1})} + \frac{1}{\lambda}D_{h}\left(x^{k} , x^{k - 1}\right) & \leq f (x^{\ast}) + \act{x^{\ast} - x^{k - 1} , \nabla g(x^{k - 1})} \\
       		& + \frac{1}{\lambda}D_{h}\left(x^{\ast} , x^{k - 1}\right),
	  	\end{align*}
	  	that is,
		\begin{equation*}
       		f(x^{k}) \leq f\left(x^{\ast}\right) + \act{x^{\ast} - x^{k} , \nabla g\left(x^{k - 1}\right)} + \frac{1}{\lambda}D_{h}\left(x^{\ast} , x^{k - 1}\right) - \frac{1}{\lambda}D_{h}\left(x^{k} , x^{k - 1}\right).
	  	\end{equation*}
		Substituting $k$ by $n_{k}$ and letting $k \rightarrow \infty$, we obtain from Proposition \ref{P:SuffDesc}(iii) and the facts mentioned above, that
		\begin{equation*}
       		\limsup_{k \rightarrow \infty} f\left(x^{n_{k}}\right) \leq f\left(x^{\ast}\right).
	  	\end{equation*}
	  	This proves condition (C3), and thanks to Lemma \ref{L:SubConv}, the first item of the theorem follows. Moreover, we also obtain that $\Psi$ is constant on $\omega(x^0)$.
\medskip

		(ii) Since $\Psi$ is assumed to have the KL property, and we have just shown that $\Psi$ is constant on $\omega(x^0)$, we can then apply Theorem \ref{T:AbstrGlob} given in Appendix \ref{A:GlobConv}, with $F:=\Psi$. This proves that the sequence $\Seq{x}{k}$ has finite length, and is globally  convergent  to a critical point $x^{\ast}$ of $\Psi$.
	\end{proof}
	We end by noting that, when the data $f$ and $g$ are semi-algebraic, the convergence rate of the generated  sequence described in Theorem \ref{T:GlobConv} follows by invoking Theorem \ref{th:rate}, see the Appendix \ref{A:GlobConv} for details.

\section{Application to Quadratic Inverse Problems} \label{Sec:Appl}

\subsection{Motivation and Problem's Statement}
	This section illustrates the potential of our approach and results. To this end, we consider the broad class of problems which consists of solving approximately a system of {\em quadratic} equations and can be described as follows. Given a finite number of symmetric matrices $A_{i} \in \real^{d \times d}$, $ i = 1 , 2 , \ldots , m$, describing the model, and a vector of possibly noisy measurements $b \in \rr^{m}$, the goal is to find $x \in \real^{d}$, that solves the following system
	\begin{equation*}
		x^{T}A_{i}x \simeq b_{i}, \quad i = 1 , 2 , \ldots , m.
	\end{equation*}
	This type of problems is a natural extension of the classical linear inverse problem which arise in the broad area of signal recovery, whereby linear measurements are replaced by {\em quadratic} measurements, see \cite{BY12} and references therein. It also includes the well-known and fundamental class of phase retrieval problems as a special case, which has been, and is still intensively studied  in the literature; see \cite{L2017} for a very recent review on this problem and references therein.
\medskip
	
	Commonly, the system under consideration is underdetermined, thus some prior information on the original signal is incorporated into the model by means of some regularizer represented by a function $f$, possibly nonconvex, nonsmooth, and extended valued (to accommodate constraints). Adopting the usual least-squares model to measure the error, the problem can then be reformulated as the nonconvex optimization problem
	\begin{equation*}
		{\rm (QIP)} \qquad \min  \left\{ \Psi\left(x\right) := \frac{1}{4}\sum_{i = 1}^{m} \left(x^{T}A_{i}x - b_{i}\right)^{2} + \theta f\left(x\right) : \, x \in \real^{d} \right\},
	\end{equation*}
	where $\theta > 0$ plays the role of a penalty parameter controlling the trade-off between matching the data fidelity criteria versus its regularizer $f$.
\medskip

	Clearly, the nonconvex function $g : \rr^{d} \rightarrow \erl$ defined by $g\left(x\right) := \left(1/4\right)\sum_{i = 1}^{m} \left(x^{T}A_{i}x - b_{i}\right)^{2}$ is continuously differentiable on $\real^{d}$, and it {\em does not} admit a global Lipschitz continuous gradient, thus
precluding the use of the classical proximal gradient method.
\medskip

	In the following, we demonstrate the power of our  approach, illustrating how new and simple globally convergent schemes can be derived for the broad class of problems (QIP). In particular, we design two new  algorithms: one for the unconstrained $\ell_{1}$ regularized (QIP) model, and the other for the sparsity $\ell_{0}$ constrained (QIP); both models being known to promote sparse solutions. Our algorithms are given through an explicit {\em closed form formula}, and are thus straightforward to implement. To the best of our knowledge, they appear to be the first globally convergent schemes for this class of problems.

\subsection{Two Algorithms for Sparse Quadratic Inverse Problems}
	Throughout this section, the underlying space is $C \equiv \real^{d}$. Given the nonconvex function $g :\real^{d} \rightarrow \real$ defined by
	\begin{equation*}
		g\left(x\right) = \frac{1}{4}\sum_{i = 1}^{m} \left(x^{T}A_{i}x - b_{i}\right)^{2},
	\end{equation*}
	we consider solving the nonconvex model (QIP) in the following two cases:
	\begin{itemize}
		\item[(a)] \textbf{A convex $\ell_{1}$-norm regularization.} $f : \real^{d} \rightarrow \real$ with $f\left(x\right) = \norm{x}_{1}$.
		\item[(b)] \textbf{A nonconvex $\ell_{0}$-ball constraint.} $f : \real^{d} \rightarrow \erl$ with
$f\left(x\right) = \delta_{\mathbb{B}_{0}^{s}}\left(x\right)$, the indicator of the $\ell_{0}$-ball, where for a positive integer $s < d$,
			\begin{equation*}
				\mathbb{B}_{0}^{s} \equiv \left\{ x : \; \norm{x}_{0} \leq s \right\},
			\end{equation*}
	 		and $\norm{x}_{0}$ is the ``$\ell_{0}$ norm" which counts the number of nonzero components in $x$.
	\end{itemize}
 	To apply our approach on the (QIP) model in both cases (a) and (b), we first need to identify a suitable function $h \in \mathcal{G}(\real^{d})$ such that \textbf{L-smad} holds for the pair $\left(g , h\right)$. Here, we use $h : \real^{d} \rightarrow \real$ given by
 	\begin{equation*}
 		h\left(x\right) = \frac{1}{4}\norm{x}_{2}^{4} + \frac{1}{2}\norm{x}_{2}^{2}.
 	\end{equation*}
	Equipped with this $h$, we now show that \textbf{L-smad} holds, \ie there exists $L > 0$ such that $Lh - g$ is convex on $\real^{d}$.
	\begin{lemma} \label{L:ConsL}
		Let $g$ and $h$ as defined above. Then, for any $L$ satisfying
		\begin{equation*}
			L \geq \sum_{i = 1}^{m} \left(3\norm{A_{i}}^{2} + \norm{A_{i}}\left|b_{i}\right|\right),
		\end{equation*}
		the function $Lh - g$ is convex on $\real^{d}$.
	\end{lemma}
	\begin{proof}
		Let $x \in \real^{d}$. Since $g$ and $h$ are $C^{2}$ on $\real^{d}$, in order to warrant the convexity of $Lh - g$ it is sufficient to find $L > 0$ such that $L\lambda_{\min}\left(\nabla^{2} h\left(x\right)\right) \geq \lambda_{\max}\left(\nabla^{2} g\left(x\right)\right)$, where $\lambda_{\min}(M)$ and $\lambda_{\max}(M)$ denote the minimal and maximal eigenvalue of a matrix $M$, respectively.
\medskip
		
		By a straightforward computation we obtain that
		\begin{equation*}
			\nabla^{2} g\left(x\right) = \sum_{i = 1}^{m} \left(2A_{i}xx^{T}A_{i} + \left(x^{T}A_{i}x - b_{i}\right)A_{i}\right)\quad \text{and}\quad 	\nabla^{2} h\left(x\right) = \left(\norm{x}_{2}^{2} + 1\right)I_{d} + 2xx^{T},
		\end{equation*}	
		and due to the well-known fact that $\lambda_{\max}(M)  \leq \norm{M}$ we can write
		\begin{equation*}
			\lambda_{\max}\left(\nabla^{2} g\left(x\right)\right) \leq \norm{\nabla^{2}g\left(x\right)} \leq\sum_{i = 1}^{m} \left(3\norm{A_{i}}^{2}\norm{x}^{2} + \norm{A_{i}}\left|b_{i}\right|\right).
		\end{equation*}
		On the other hand, since $\nabla^{2} h\left(x\right) \succeq \left(\norm{x}_{2}^{2} + 1\right)I_{d} $, we obtain that $\lambda_{\min}\left(\nabla^{2} h\left(x\right)\right) \geq \norm{x}_{2}^{2} + 1$. Therefore, taking $L \geq \sum_{i = 1}^{m} \left(3\norm{A_{i}}^{2} + \norm{A_{i}}\left|b_{i}\right|\right)$ yields
		\begin{equation*}
			\lambda_{\max}\left(\nabla^{2} g\left(x\right)\right) \leq \sum_{i = 1}^{m} \left(3\norm{A_{i}}^{2}\norm{x}^{2} + \norm{A_{i}}\left|b_{i}\right|\right) \leq L\left(\norm{x}_{2}^{2} + 1\right) \leq L \lambda_{\min}\left(\nabla^{2} h\left(x\right)\right),
		\end{equation*}
 		which proves the desired result.
	\end{proof}	
 	In order to apply our results (\cf Section \ref{ss:global}) for both scenarios (a) and (b), we observe   that $h$ given above is $1$-strongly convex on $\real^{d}$ and it easy to see that Assumptions\ref{A:AssumptionA}-\ref{A:AssumptionD} hold. Furthermore, the function $g$ is a real polynomial, hence semi-algebraic, and both functions $\norm{x}_{0}$ and $\norm{x}_{1}$ is also semi-algebraic (see, \cite[Appendix 5, p. 490]{BST2014}). Therefore, since the addition of semi-algebraic functions results in a  semi-algebraic function, it follows that for both models (a) and (b), the objective $\Psi$ is semi-algebraic, and BPG can be applied on the (QIP) model in each case, to produce a globally convergent sequence which converges to a critical point of $\Psi$.
\medskip

	To apply BPG, the main computational step require us to compute the Bregman proximal gradient map (see \eqref{D:OperT}):
	\begin{equation*}
		T_{\lambda}\left(x\right) = \argmin \left\{ f\left(u\right) + \act{\nabla g\left(x\right) , u - x} + \frac{1}{\lambda} D_{h}\left(u , x\right) : \, u \in \real^{d} \right\} \; (\lambda >0).
	\end{equation*}
	We now show that for both scenarios (a) and (b), this step yields an {\em explicit} closed form formula.
\medskip

	Before doing so, we introduce some convenient notations and recall some well known operators that will be used in the rest of this section. Let $\lambda > 0$ and fix any $x \in \real^{d}$. Define
	\begin{equation}\label{pl}
 		p \equiv p_{\lambda}\left(x\right) = \lambda \nabla g\left(x\right) - \nabla h\left(x\right) \quad (\text{for simplicity we often drop}\; \lambda , x).
 	\end{equation}
 	For the pair $(g , h)$ given above, a direct computation of their gradients gives $p_{\lambda}\left(x\right)$. Now, omitting constant terms in the above expression defining $T_{\lambda}$, we can write
	\begin{equation}\label{simpT}
		T_{\lambda}\left(x\right) = \argmin \left\{ \lambda f\left(u\right) + \act{p_{\lambda}\left(x\right) , u} + h\left(u\right) : \; u \in \real^{d} \right\}.
	\end{equation}
	We now recall two very well-known operators that will be used to compute the resulting $T_{\lambda}$.
   	\begin{itemize}
 		\item[] {\bf Soft-thresholding (with parameter $\tau$).} For any $y \in \real^{d}$,
			\begin{equation} \label{eq:soft}
				\SSS_{\tau}\left(y\right) = \argmin_{x \in \real^{d}} \left\{ \tau\norm{x}_{1} + \frac{1}{2}\norm{x - y}^{2} \right\} = \max\left\{ \left|y\right| - \tau , 0 \right\}\sgn(y),
			\end{equation} 		  		   				
			with the absolute value understood to be component-wise.
 		\item[] {\bf Hard-thresholding (with parameter $\tau$).} For any $y\in \real^{d}$,
			\begin{equation} \label{eq:hard}
				\HHH_{\tau}\left(y\right) = \argmin_{x \in \real^{d}} \left\{ \norm{x - y}^{2} : \; x \in \mathbb{B}_0^\tau \right\}
				= \begin{cases}
					y_i, & i\leq \tau,\\
					0, & {\rm otherwise,}
				\end{cases}
		\end{equation} 		
		where we assumed, without the loss of generality, that the vector $y$ is sorted in a decreasing order according to the absolute value of its entries (ties may be broken arbitrarily), \ie $|y_{1}| \geq |y_{2}| \geq \cdots \geq |y_{n}|$.						   		 		   	
   	\end{itemize}
	We are now ready to establish the promised formula of $T_{\lambda}$ for the two cases (a) and (b).
	\begin{proposition}[Bregman Proximal Formula for the $\ell_{1}$-Norm Regularization] \label{prop:prox_comp_l1}
		Let $f = \norm{\cdot}_{1}$, and for $x \in \real^{d}$, let $v(x) := \SSS_{\lambda\theta}\left(p_{\lambda}\left(x\right)\right)$. Then, $x^{+} = T_{\lambda}\left(x\right)$ is given by
		\begin{equation*}
			x^{+} = - t^{\ast}v(x) = t^{\ast}\SSS_{\lambda\theta}\left(\nabla h\left(x\right) - \lambda\nabla g\left(x\right)\right),
		\end{equation*}
		where $t^{\ast}$ is the unique positive real root of
		\begin{equation*}
			t^{3}\norm{v(x)}^{2}_{2} + t - 1 = 0,
		\end{equation*}
		which admits an explicit formula.\footnote{The general solution of a cubic equation was apparently discovered first by Scipione del Ferro and published by Gerolamo Cardano in his {\em Ars magna} in 1545, see \url{https://www.britannica.com/topic/cubic-equation}}.
	\end{proposition}
	\begin{proof}
		From the formulation of $T_{\lambda}$ given in \eqref{simpT}, we have
		\begin{equation} \label{eq:bregman_prox}
			x^{+} = \argmin \left\{ \lambda\theta\norm{u}_{1} + \act{p , u} + \frac{1}{4}\norm{u}_{2}^{4} + \frac{1}{2}\norm{u}_{2}^{2} : \; u \in \real^{d} \right\},
		\end{equation}		
		where we recall that $p = \lambda\nabla g\left(x\right) - \nabla h\left(x\right)$. By the first order global optimality condition for the strongly convex problem \eqref{eq:bregman_prox} we obtain for each $i = 1 , 2 , \ldots , d$,
		\begin{equation*}
			x_{i}^{+}\left(1 + \norm{x^{+}}_{2}^{2}\right) + p_{i} + \lambda\theta\gamma_{i} = 0,
		\end{equation*}
		where $\gamma_{i} = \sgn(x_{i}^{+})$ if $x_{i}^{+} \neq 0$, and $\gamma_{i} \in \left[-1 , 1\right]$ if $x_{i}^{+} = 0$. For each $i = 1 , 2 , \ldots , d$, set $v_{i} = p_{i} + \lambda\theta\gamma_{i}$. Then,  from the above equation, $x_{i}^{+} = -tv_{i}$ with $t > 0$, and
		\begin{equation*}
			v_{i}\left(t^{3}\norm{v}_{2}^{2} + t - 1\right) = 0, \; i = 1 , 2 , \ldots , d.
		\end{equation*}
		Now consider the following two cases for any $i = 1 , 2 , \ldots , d$:
		\begin{itemize}
			\item If $v_{i} = 0$, then $x_{i}^{+} = 0$, and $\gamma_{i} \in \left[-1 , 1\right]$, \ie we obtain $x_{i}^{+} = 0$ whenever $\left|p_{i}\right| \leq \lambda\theta$.
			\item If $v_{i} \neq 0$, then $x_{i}^{+} \neq 0$, and $\gamma_{i} = \sgn(x_{i}^{+}) = \sgn(-t v_{i}) = -\sgn(v_{i})$. Therefore, using $\gamma_{i} = \left(v_{i} - p_{i}\right)/\left(\lambda\theta\right)$, the latter relation reads $\lambda\theta\sgn(v_{i}) + v_{i} - p_{i} = 0$, which is nothing else but the following optimality condition for $v$:
				\begin{equation*}
					v = \argmin \left\{ \lambda\theta\norm{\xi}_{1} + \frac{1}{2}\norm{\xi - p}_{2}^{2} : \; \xi \in \real^{d} \right\} \equiv \SSS_{\lambda\theta}\left(p\right).
				\end{equation*}
				where the last equality is from the definition of the soft threshold operator (see \eqref{eq:soft}.
		\end{itemize}
		To summarize, in all cases, we have thus obtained that $x^{+} = - t^{\ast}\SSS_{\lambda\theta}\left(p\right)$, where $t^{\ast}$ is the positive root of the cubic equation $t^{3}\norm{v}_{2}^{2} + t - 1 = 0$, which is given by a closed form expression.
	\end{proof}
	Now, we turn to the $\ell_{0}$-norm constrained version. First we recall the following useful result established in  \cite[Proposition 4.3 p. 79]{luss2013conditional}.
	\begin{lemma}\label{l:trunc}
		Given $0 \neq a \in \real^{d}$ and a positive integer $s < d$, then
		\begin{equation*}
			\max \left\{ \act{a , z} : \; \norm{z}_{2} = 1, \; \norm{z}_{0} \leq s \right\} = \norm{\HHH_{s}\left(a\right)}_{2},
		\end{equation*}
		with optimal solution obtained at $z^{\ast} = \HHH_{s}\left(a\right)/\norm{\HHH_{s}\left(a\right)}_{2}$.
	\end{lemma}
	\begin{proposition}[Bregman Proximal Formula for the $\ell_{0}$-Ball Constraint] \label{prop:prox_comp_l0}
		Let $f = \delta_{\mathbb{B}_{0}^{s}}$ and $x \in \real^{d}$. Then, $x^{+} = T_{\lambda}\left(x\right)$ is given by
		\begin{equation*}
			x^{+} = \sqrt{t^{\ast}}\norm{\HHH_{s}\left(p_{\lambda}\left(x\right)\right)}_{2}^{-1}\HHH_{s}\left(p_{\lambda}\left(x\right)\right),
		\end{equation*}
		where $\sqrt{t^{\ast}} \equiv \eta^{\ast}$ is the unique nonnegative real root of the cubic equation
		\begin{equation} \label{eq:cubic_l0}
			\eta^{3} + \eta - \norm{\HHH_{s}\left(p_{\lambda}\left(x\right)\right)}_{2} = 0.
		\end{equation}
	\end{proposition}
	\begin{proof}
		We use the same notations as in Proposition \ref{prop:prox_comp_l1}. Here $x^{+} = T_{\lambda}\left(x\right)$ consists in  finding $x^{+}$ which solves the nonconvex problem:
		\begin{equation*}
			\min_{ u \in \real^{d}} \left\{ \act{p , u} + \frac{1}{4}\norm{u}_{2}^{4} + \frac{1}{2}\norm{u}_{2}^{2} : \; \norm{u}_{0} \leq s \right\}.
		\end{equation*}
		Introducing the new variable $t := \norm{u}_{2}^{2} \geq 0$, the latter is equivalent to solving
		\begin{equation} \label{eq:equiv_primal_l0}
			\min_{t \in \real_{+}} \left\{ \frac{1}{4}t^{2} + \frac{1}{2}t + \min_{u \in \real^{d}} \left\{ \act{p , u} : \; \norm{u}_{0} \leq s, \, \norm{u}_{2}^{2} = t \right\} \right\}.
		\end{equation}
		Invoking Lemma \ref{l:trunc} for the inner minimization with respect to $u$, one easily obtains:
		\begin{equation*}
			\min_{u \in \real^{d}} \left\{ \act{p , u} : \; \norm{u}_{0} \leq s, \; \norm{u}_{2}^{2} = t \right\} = -\sqrt{t}\HHH_{s}\left(-p\right)\equiv -\sqrt{t}\HHH_{s}\left(p\right),
		\end{equation*}
		with solution, $u^{\ast} = \sqrt{t}\norm{\HHH_{s}\left(p\right)}_{2}^{-1}\HHH_{s}\left(p\right)$ where we used the fact that for the hard threshold operator: $\HHH_{s}\left(p\right) = \HHH_{s}\left(-p\right)$). Therefore, problem \eqref{eq:equiv_primal_l0} reduces to solve
		\begin{equation} \label{eq:varphi_opt}
			\min_{t \in \real_{+}} \left\{ \psi\left(t\right) := -\norm{\HHH_{s}\left(p\right)}_{2}\sqrt{t} + \frac{1}{4}t^{2} + \frac{1}{2}t \right\},
		\end{equation}
		and it follows that $x^{+} = \sqrt{t^{\ast}}\norm{\HHH_{s}\left(p\right)}_{2}^{-1}\HHH_{s}\left(p\right)$, where $t^{\ast}$ is an optimal solution of \eqref{eq:varphi_opt}, thus proving the first part of the proposition. Now, if $\norm{\HHH_{s}\left(p\right)}_{2} = 0$ then clearly $t^{\ast} = 0$. Otherwise, with $\norm{\HHH_{s}\left(p\right)}_{2} > 0$, since $\lim_{t \rightarrow 0^{+}} \psi'\left(t\right) = -\infty$ and the above objective $\psi$ is strongly convex over $\real_{+}$, the problem \eqref{eq:varphi_opt} admits a unique minimizer $t^{\ast} > 0$ which satisfies the optimality condition:
		\begin{equation*}
			\sqrt{t}\left(t + 1\right) = \norm{\HHH_{s}\left(p\right)}_{2}.
		\end{equation*}
		Thus, in either cases, the unique optimal solution $t^{\ast}$ for problem \eqref{eq:varphi_opt} is the non-negative real root of the last equation, which can be rewritten as \eqref{eq:cubic_l0}, and hence the desired result is proved.
	\end{proof}
	Equipped with Lemma \ref{L:ConsL} together with Propositions \ref{prop:prox_comp_l1} and \ref{prop:prox_comp_l0}, thanks to Theorem \ref{T:GlobConv}, we have now all the ingredients to formulate simple explicit algorithms in terms of the problem's data for solving a broad class of sparse quadratic inverse problems (SQIP). A thorough computational study of (SQIP), and some other related extensions will be developed in a separate paper.

\section{Appendix: Global Convergence for KL Functions} \label{A:GlobConv}
	The objective of this appendix is twofold. First, to make the paper self-contained, and second to outline here the general convergence mechanism as recently described in \cite{BST2014} on the so-called PALM algorithm, so that it can be used and applied to any given algorithm as described below. To this end, let  $F : \real^{d} \rightarrow \erl$  be a proper and lower semicontinuous function which is bounded from below and consider the problem
    \begin{equation*}
        \inf \left\{ F\left(x\right) : \; x \in \real^{d} \right\}.
    \end{equation*}
	Consider a generic algorithm $\AAA$ which generates a sequence $\Seq{x}{k}$ via the following:
    \begin{equation*}
   		\text{start with any}\, x^{0} \in \real^{d} \,\, \text{and set} \,\, x^{k + 1} \in \AAA\left(x^{k}\right), \quad k = 0 , 1 , \ldots.
    \end{equation*}
    The main goal is to prove that the {\em whole sequence} $\Seq{x}{k}$, generated by the algorithm $\AAA$, converges to a critical point of $F$.
For that purpose we first outline the three key ingredients of the forthcoming methodology.\footnote{To keep this appendix as a completely independent unit, we repeat here a definition and results already stated in Section 4.}
	\begin{definition}[Gradient-like descent sequence] \label{def-descent}
		A sequence $\Seq{x}{k}$ is called \textit{a gradient-like descent sequence} for $F$ if the following three conditions hold:
		\begin{itemize}
        		\item[$\rm{(C1)}$] \textit{Sufficient decrease property.} There exists a positive scalar $\rho_{1}$ such that
            		\begin{equation*}
                		\rho_{1}\norm{x^{k + 1} - x^{k}}^{2} \leq F(x^{k}) - F(x^{k + 1}), \quad \forall \,\, k \in \nn.
	            \end{equation*}
    			\item[$\rm{(C2)}$] \textit{A subgradient lower bound for the iterates gap.} There exists a positive scalar $\rho_{2}$ such that         	
		   		\begin{equation*}
	    			\norm{w^{k + 1}} \leq \rho_{2}\norm{x^{k + 1} - x^{k}}, \quad w^{k + 1} \in \partial F(x^{k + 1}), \quad \forall \,\, k \in \nn.
	            \end{equation*}
			\item[$\rm{(C3)}$] Let $\overline{x}$ be a limit point of a subsequence $\left\{ x^{k} \right\}_{k \in {\cal K}}$, then $\limsup_{k \in {\cal K} \subset \nn} F\left(x^{k}\right) \leq F\left(\overline{x}\right)$.
        \end{itemize}	
	\end{definition}
	The two conditions (C1) and (C2) are typical for any descent type algorithm (see, \eg \cite{AB2009}), and basic to prove subsequential convergence. The condition (C3) is a minimal and weak requirement, which, in particular, holds when $F$ is continuous.
\medskip

	The set of all limit points of $\Seq{x}{k}$ is defined by
    \begin{equation*}
    		\omega\left(x^{0}\right) := \left\{ \overline{x} \in \real^{d}: \; \exists \mbox{ an increasing sequence of integers } \seq{k}{l} \mbox{ such that }\; x^{k_{l}} \rightarrow \overline{x} \mbox{ as } l \rightarrow \infty \right\}.
    \end{equation*}	
	The next result establish the promised subsequential convergence.
	\begin{lemma}[Subsequence Convergence] \label{L:SubConv1}
		Let $\Seq{x}{k}$ be a gradient-like descent sequence for $F$ which is assumed to be bounded. Then, $\omega\left(x^{0}\right)$ is a nonempty and compact subset of $\crit F$, and we have
		\begin{equation} \label{L:SubConv:1}
			\limit{k}{\infty} \dist\left(x^{k} , \omega\left(x^{0}\right)\right) = 0.
		\end{equation}
		In addition, the objective function $F$ is finite and constant on $\omega\left(x^{0}\right)$.
	\end{lemma}
	\begin{proof}
		Since $\Seq{x}{k}$ is bounded there is $x^{\ast} \in \real^{d}$ and a subsequence $\left\{ x^{k_{q}} \right\}_{q \in \nn}$ such that $x^{k_{q}} \rightarrow x^{\ast}$ as $q \rightarrow \infty$ and hence $\omega\left(x^{0}\right)$ is nonempty. Moreover, the set $\omega\left(x^{0}\right)$ is compact since it can be viewed as an intersection of compact sets. Now, from condition (C3) and the lower semicontinuity of $F$, we obtain
		\begin{equation} \label{L:SubConv:2}
       		\lim_{q \rightarrow \infty} F\left(x^{k_{q}}\right) = F\left(x^{\ast}\right).
      	\end{equation}
       	On the other hand, from conditions (C1) and (C2), we know that there is $w^{k} \in \partial F\left(x^{k}\right)$, $k \in \nn$, such that $w^{k} \rightarrow 0$ as $k \rightarrow \infty$. The closedness property\footnote{Let $\left\{ \left(x^{k} , u^{k}\right) \right\}_{k \in \nn}$ be a sequence in $\gr{\left(\partial F\right)}$ that converges to $\left(x , u\right)$ as $k \rightarrow \infty$. By the very definition of $\partial F\left(x\right)$, if $F(x^{k})$ converges to $F(x)$ as $k \rightarrow \infty$, then $\left(x , u\right) \in \gr{(\partial F)}$.} of $\partial F$ implies thus that $0 \in \partial F\left(x^{\ast}\right)$. This proves that $x^{\ast}$ is a critical point of $F$, and hence \eqref{L:SubConv:1} is valid.

		To complete the proof, let $\limit{k}{\infty} F\left(x^{k}\right) = l \in \real$. Then $\left\{ F\left(x^{k_{q}}\right) \right\}_{q \in \nn}$ converges to $l$ and from \eqref{L:SubConv:2} we have $F\left(x^{\ast}\right) = l$. Hence the restriction of $F$ to $\omega\left(x^{0}\right)$ equals $l$.
	\end{proof}
	To achieve our main goal, \ie to establish global convergence of the \textit{whole} sequence, we need an additional assumption on the class of functions $F$: it must satisfy the so-called nonsmooth Kurdyka-{\L}ojasiewicz (KL) property \cite{BDL2006}(see \cite{K1998,L1963} for smooth cases). We refer the reader to \cite{BDLM10} for an in depth study of the class of KL functions, as well as references therein. We provide now the formal definition of the KL property and two important results.
\medskip

	Denote $[\alpha < F < \beta] := \left\{ x \in \real^{d} : \; \alpha < F\left(x\right) < \beta \right\}$. Let $\eta > 0$, and set
	\begin{equation*}
		\Phi_{\eta} = \left\{ \varphi \in C^{0}[0 , \eta) \cap C^{1}(0 , \eta) : \; \varphi\left(0\right) = 0, \varphi \; \text{concave and} \; \varphi' > 0 \right\}.
	\end{equation*}
	\begin{definition}[The nonsmooth KL property] \label{def:kl}
		A proper and lower semicontinuous function $F : \real^{d} \rightarrow \erl$ has the Kurdyka-{\L}ojasiewicz (KL) property locally at $\overline{u} \in \dom F$ if there exist $\eta > 0$, $\varphi \in \Phi_{\eta}$, and a neighborhood $U\left(\overline{u}\right)$ such that
		\begin{equation*}
  			\varphi'\left(F\left(u\right) - F\left(\overline{u}\right)\right)\dist\left(0 , \partial F\left(u\right)\right) \geq 1,
        \end{equation*}
		for all $u \in U\left(\overline{u}\right) \cap \left[F\left(\overline{u}\right) < F\left(u\right) < F\left(\overline{u}\right) + \eta \right]$.
  	\end{definition}
	Verifying the KL property of a given function might often be a difficult task. However, thanks to a fundamental result established in \cite{BDL2006}, it holds for the broad class of \textit{semi-algebraic} functions.
    \begin{theorem} \label{P:Semi-aKL}
        Let $F : \real^{d} \rightarrow \erl$ be a proper and lower semicontinuous function. If $F$ is semi-algebraic then it satisfies the KL property at any point of $\dom{F}$.
    \end{theorem}
	Our last ingredient is a key uniformization of the KL property proven in \cite[Lemma 6, p. 478]{BST2014}, which we record below.
    \begin{lemma}[Uniformized KL property] \label{L:KLProperty}
        Let $\Omega$ be a compact set and let $F : \real^{d} \rightarrow \erl$ be a proper and lower semicontinuous function. Assume that $F$ is constant on $\Omega$ and satisfies the KL property at each point of $\Omega$. Then, there exist $\varepsilon > 0$, $\eta > 0$ and $\varphi \in \Phi_{\eta}$ such that for all $\overline{x}$ in $\Omega$  one has,
        \begin{equation} \label{L:KLProperty:2}
            \varphi'\left(F\left(x\right) - F\left(\overline{x}\right)\right)\dist\left(0 , \partial F\left(x\right)\right) \geq 1,
        \end{equation}
        and all $x \in  \left\{ x \in \real^{d} : \; \dist\left(x , \Omega\right) < \varepsilon \right\} \cap \left[ F\left(\overline{x}\right) < F\left(x\right) < F\left(\overline{x}\right) + \eta \right]$.
            \end{lemma}	
	We can now conveniently summarize the convergence results of \cite{BST2014} through the following abstract convergence result.
	\begin{theorem}[Global convergence] \label{T:AbstrGlob}
		Let $\Seq{x}{k}$ be a bounded gradient-like descent sequence for $F$. If $F$ satisfies the KL property, then the sequence $\Seq{x}{k}$ has finite length, \ie $\sum_{k = 1}^{\infty} \norm{x^{k + 1} - x^{k}} < \infty$ and it converges to  $x^{\ast} \in \crit F$.
	\end{theorem}
    \begin{proof}
        Since $\Seq{x}{k}$ is bounded there exists a subsequence $\left\{ x^{k_{q}} \right\}_{q \in \nn}$ such that $x^{k_{q}} \rightarrow \overline{x}$ as $q \rightarrow \infty$. In a similar way as in Lemma \ref{L:SubConv} we get that
        \begin{equation} \label{T:AbstrGlob:1}
            \lim_{k \rightarrow \infty} F(x^{k}) = F\left(\overline{x}\right).
        \end{equation}
        If there exists an integer $\bar{k}$ for which $F(x^{\bar{k}}) = F\left(\overline{x}\right)$ then condition (C1) would imply that $x^{\bar{k} + 1} = x^{\bar{k}}$. A trivial induction show then that the sequence $\Seq{x}{k}$ is stationary and the announced results are obvious. Since $\left\{ F\left(x^{k}\right) \right\}_{k \in \nn}$ is a nonincreasing sequence, it is clear from \eqref{T:AbstrGlob:1} that $F\left(\overline{x}\right) < F\left(x^{k}\right)$ for all $k > 0$. Again from \eqref{T:AbstrGlob:1} for any $\eta > 0$ there exists a nonnegative integer $k_{0}$ such that $F\left(x^{k}\right) < F\left(\overline{x}\right) + \eta$ for all $k > k_{0}$. From Lemma \ref{L:SubConv} we know that $\lim_{k \rightarrow \infty} \dist\left(x^{k} , \omega\left(x^{0}\right)\right) = 0$. This means that for any $\varepsilon > 0$ there exists a positive integer $k_{1}$ such that $\dist\left(x^{k} , \omega\left(x^{0}\right)\right) < \varepsilon$ for all $k > k_{1}$.
\medskip

		From Lemma \ref{L:SubConv}, we know that $\omega\left(x^{0}\right)$ is nonempty and compact, the function $F$ is finite and constant on $\omega\left(x^{0}\right)$. Hence, we can apply the Uniformization Lemma with $\Omega = \omega\left(x^{0}\right)$. Therefore, for any $k > l := \max\left\{ k_{0} , k_{1} \right\}$, we have
       	\begin{equation} \label{T:FiniteLength-Item1:1}
        		\varphi'\left(F(x^{k}) - F(\overline{x})\right)\dist\left(0 , \partial F(x^{k})\right) \geq 1.
        \end{equation}
       	This makes sense since we know that $F\left(x^{k}\right) > F\left(\overline{x}\right)$ for any $k > l$. From condition (C2) we get that
  		\begin{equation} \label{T:FiniteLength-Item1:2}
        		\varphi'\left(F(x^{k}) - F\left(\overline{x}\right)\right) \geq \frac{1}{\rho_{2}}\norm{x^{k} - x^{k - 1}}^{-1}.
       	\end{equation}
        For convenience, we define for all $p , q \in \nn$ and $\overline{x}$ the following quantity
        \begin{equation*}
            	\Delta_{p , q} : = \varphi\left(F\left(x^{p}\right) - F\left(\overline{x}\right)\right) - \varphi\left(F\left(x^{q}\right) - F\left(\overline{x}\right)\right).
        \end{equation*}
      	From the concavity of $\varphi$ we get that
       	\begin{equation} \label{T:FiniteLength-Item1:3}
      		\Delta_{k , k + 1} \geq \varphi'\left(F(x^{k}) - F(\overline{x})\right)\left(F(x^{k}) - F(x^{k + 1})\right).
      	\end{equation}
      	Combining condition (C1) with \eqref{T:FiniteLength-Item1:2} and \eqref{T:FiniteLength-Item1:3} yields, for any $k > l$, that
       	\begin{equation} \label{T:FiniteLength-Item1:4}
        		\Delta_{k , k + 1} \geq \frac{\norm{z^{k + 1} - z^{k}}^{2}}{\rho\norm{z^{k} - z^{k - 1}}}, \quad \text{ where} \,\, \rho := \rho_{2}/\rho_{1}.
        \end{equation}
        Using the fact that $2\sqrt{\alpha\beta} \leq \alpha + \beta$ for all $\alpha , \beta \geq 0$, we infer from the later inequality that
       	\begin{equation} \label{T:FiniteLength-Item1:5}
        		2\norm{x^{k + 1} - x^{k}} \leq \norm{x^{k} - x^{k - 1}} + \rho\Delta_{k , k + 1}.
        \end{equation}
       	Summing up \eqref{T:FiniteLength-Item1:5} for $i = l + 1 , \ldots , k$ yields
		\begin{align*}
  			2\sum_{i = l + 1}^{k} \norm{x^{i + 1} - x^{i}} & \leq \sum_{i = l + 1}^{k} \norm{x^{i} - x^{i - 1}} + \rho\sum_{i = l + 1}^{k} \Delta_{i , i + 1} \\
           	& \leq \sum_{i = l + 1}^{k} \norm{x^{i + 1} - x^{i}} + \norm{x^{l + 1} - x^{l}} + \rho\sum_{i = l + 1}^{k} \Delta_{i , i + 1} \\
           	& = \sum_{i = l + 1}^{k} \norm{x^{i + 1} - x^{i}} + \norm{x^{l + 1} - x^{l}} + \rho\Delta_{l + 1 , k + 1}
       	\end{align*}
       	where the last inequality follows from the fact that $\Delta_{p , q} + \Delta_{q , r} = \Delta_{p , r}$ for all $p , q , r \in \nn$. Since $\varphi \geq 0$, recalling the definition of $\Delta_{l + 1 , k + 1}$, we thus have for any $k > l$ that
       	\begin{align*}
         	\sum_{i = l + 1}^{k} \norm{x^{i + 1} - x^{i}} \leq \norm{x^{l + 1} - x^{l}} + \rho\varphi\left(F(x^{l + 1}) - F\left(\overline{x}\right)\right),
        \end{align*}
        which implies that $\sum_{k = 1}^{\infty} \norm{x^{k + 1} - x^{k}} < \infty$, \ie $\Seq{x}{k}$ is a Cauchy sequence and hence together with Lemma~\ref{L:SubConv}, we obtain the global convergence to a critical point is established.
    \end{proof}
   	Finally, we record a generic rate of convergence result (see \cite{AB2009}) which is also valid under the same assumptions as Theorem \ref{T:AbstrGlob}.
    \begin{theorem}[Rate of convergence]\label{th:rate}
	    Let $\Seq{x}{k}$ be a bounded gradient-like descent sequence for $F$. Assume that $F$ satisfies the KL property where the desingularizing function $\varphi$ of $F$ is of the following form
	    \begin{equation*}
	    		\varphi\left(s\right) = cs^{1 - \theta}, \quad c > 0, \,\, \theta \in \left[0 , 1\right).
	    \end{equation*}
	    Let $\overline{x}$ be the limit point of $\Seq{x}{k}$. Then the following estimations hold:
    		\begin{itemize}
        		\item[$\rm{(i)}$] If $\theta = 0$ then the sequence $\Seq{x}{k}$ converges to $\overline{x}$ in a finite number of steps.
        		\item[$\rm{(ii)}$] If $\theta \in \left(0 , 1/2\right]$ then there exist $\omega > 0$ and $\tau \in \left[0 , 1\right)$ such that
				\begin{equation*}        	
        			\norm{x^{k} - \overline{x}} \leq \omega\, \tau^{k}.
        		\end{equation*}
        		\item[$\rm{(iii)}$] If $\theta \in \left(1/2 , 1\right)$ then there exist $\omega > 0$ such that
            	\begin{equation*}
                	\norm{x^{k} - \overline{x}} \leq \omega\, k^{-\frac{1 - \theta}{2\theta - 1}}.
            	\end{equation*}
    		\end{itemize}
	\end{theorem}
		
	\bibliographystyle{abbrv}
	\bibliography{refs-sp}

\end{document}